\documentclass[3p]{elsarticle}

\usepackage[utf8]{inputenc}
\usepackage[T1]{fontenc}
\usepackage{amsmath, amsthm, amssymb,amsxtra,color,calligra,mathrsfs,comment,url, accents} 
\usepackage{mathtools} 
\usepackage{extpfeil}   
\usepackage[all]{xy}
\usepackage{soul}  
\usepackage{tikz,tikz-cd,enumerate}
\usetikzlibrary{matrix,arrows,decorations.pathmorphing}

\usepackage{floatrow} 

\usepackage{lineno} 
\usepackage{float}
\usepackage{xcolor} 
\colorlet{mdtRed}{red!50!black}
\colorlet{dblue}{blue!50!black}
\usepackage[colorlinks,pagebackref=true]{hyperref}
\hypersetup{linkcolor=dblue,citecolor=dblue,filecolor=dullmagenta,urlcolor=mdtRed}
\renewcommand*{\backref}[1]{}
\renewcommand*{\backrefalt}[4]{[{%
		\ifcase #1 Not cited.%
		\or $\uparrow$~#2.%
		\else $\uparrow$~#2.%
		\fi%
	}]}

\newcommand{\pz}[1]{{\mathbb{Z}}_{+}^{#1}}
\newcommand{\nz}[1]{{\mathbb{Z}}^{#1}}

\newtheorem{theorem}{Theorem}[section]
\newtheorem{proposition}[theorem]{Proposition}
\newtheorem{corollary}[theorem]{Corollary}
\newtheorem{lemma}[theorem]{Lemma}
\newdefinition{remark}{Remark}[section]
\numberwithin{equation}{subsection}
\newtheorem{example}[theorem]{Example}

\newcommand{\proofofref}{}
\newproof{zproofof}{Proof of \proofofref}

\begin{document} 
	
	\title{On Serre dimension of monoid algebras and Segre extensions} 
	
	\author{Manoj K. Keshari}
	\ead{keshari@math.iitbac.in}

	\author{Maria A. Mathew\corref{cor1}}
	\ead{maria.math@iitb.ac.in}
	\address{Department of Mathematics, Indian Institute of Technology Bombay, Powai, Mumbai 400076, India.}
	
	\cortext[cor1]{Corresponding author}
	
	\begin{abstract}
		Let $R$ be a commutative noetherian ring of dimension $d$ and $M$ be a commutative$,$ cancellative$,$ torsion-free monoid of rank $r$. Then $S$-$dim(R[M]) \leq max\{1, dim(R[M])-1 \} =  max\{1, d+r-1 \}$. Further$,$ we define a class of monoids $\{\mathfrak{M}_n\}_{n \geq 1}$ such that if $M \in \mathfrak{M}_n$ is seminormal$,$ then $S$-$dim(R[M]) \leq dim(R[M]) - n= d+r-n,$ where $1 \leq n \leq r$. As an application, we prove that for the Segre extension $S_{mn}(R)$ over $R,$ $S$-$dim(S_{mn}(R)) \leq dim(S_{mn}(R)) - \Big[\frac{m+n-1}{min\{m,n\}}\Big] =  d+m+n-1 - \Big[\frac{m+n-1}{min\{m,n\}}\Big]$.
	\end{abstract}
	
	\maketitle
	
	
	\section{Introduction}
	
	In the search for an answer to his conjecture$,$ Serre \cite{serre-MR0177011} gave a splitting theorem which states that \textit{for a commutative noetherian ring $R$ of dimension $d,$ if the rank of an $R$-projective module $P$ exceeds $d,$ then $P$ admits a decomposition with a free direct summand}. This shrinks the class of projective $R$-modules one needs to study$,$ to projective modules of rank $\leq d$. Such a decomposition of $P$ is possible if there exists a $p \in P$ and a $\phi \in Hom_R(P,R)$ such that $\phi(p) = 1$. These elements are called \textit{unimodular elements} of $P$ and $Um(P)$ denotes the set of such elements. The \textit{Serre dimension} of $R,$ written as $S$-$dim(R),$ is defined to be the smallest integer $s,$ such that if $rank(P) \geq s+1,$ then $Um(P) \neq \emptyset$.
	
	Serre's splitting theorem thus gives $S$-$dim(R) \leq dim(R)$. Plumstead \cite{Plumstead-MR722004} proved $S$-$dim(R[X]) \leq dim(R)$. Bhatwadekar-Roy \cite{Bhat-Roy-MR727374} generalized the said splitting theorem to polynomial rings $R[X_1, \ldots, X_m]$ and in \cite{BLR-MR796196}$,$ Bhatwadekar$,$ et. al. extended this result further to Laurent polynomial rings $R[X_1, \ldots, X_m, Y_1^{\pm{1}}, \ldots, Y_n^{\pm{1}}]$. Lindel \cite{Lindel-MR1322406} gave an independent proof of the same by employing semi-linear maps on the graded structure of such rings.  
	
	In another direction$,$ Weimers \cite{Weimers-unimodularMR1176904} proved the result for discrete Hodge algebras. When $R$ is a PID$,$ the corresponding result was proved by Gubeladze \cite{Gubeladze2-MR937805} for monoid algebras $R[M],$ where $M$ is a commutative$,$ cancellative$,$ torsion-free and seminormal monoid. He further conjectured in \cite{Gubeladze_toric-MR1693095}$,$ the existence of unimodular elements in a general setup $R[M]$. Swan in \cite{Swan-MR1144038}$,$ proved such an existence for any Dedekind domain $R,$ when $rank(P) \geq 2$ and $M$ is a commutative$,$ cancellative and torsion-free monoid. Keshari-Sarwar in \cite{MKK&HPS-MR3647151}$,$ gave an affirmative answer to the same for a certain class of monoids $\mathcal{C}(\phi),$ which covered the case for positive rank $2$ normal monoids. We prove the splitting theorem in the top rank case (see Theorem \ref{mmt1}):
	
	\begin{theorem}\label{mt1}
		Let R be a ring of dimension d and $M$ be a monoid of rank $r \geq 1$. Then $S$-$dim(R[M]) \leq max\{ 1, dim(R[M])-1\} = max \{1, d + r - 1\}$.
	\end{theorem}
	
	In Lemma \ref{l2}$,$ we show the existence of a positive submonoid $V$ of $M$ such that $M = U(M)V$ and study $S$-$dim(R[M])$. As a consequence$,$ for rank 2 normal monoids (not necessarily positive) we obtain $S$-$dim(R[M]) \leq dim(R)$. To tackle the case when $rank(P) < dim(R[M]),$ we define a descending chain of class of monoids $\{\mathfrak{M}_n\}_{n \geq 1},$ and prove the following (see Theorem \ref{mmt2}):
	
	\begin{theorem}\label{mt2}
		Let R be a ring of dimension d and $A=R[M],$ where $M \in \mathfrak{M}_n$ is a seminormal monoid of rank $r \geq 1$.  Assume $P$ to be a projective $A$-module of $rank > dim(A) - n = d +r-n$. Then 
		\begin{enumerate}
			\item the map $Um(P) \rightarrow Um(P/A_+^1P)$ is surjective and
			\item $S$-$dim(A) \leq dim(A) - n = d + r - n$.
		\end{enumerate}
		In particular$,$ if $M \in \mathfrak{M}_r,$ then $S$-$dim(A) \leq d $.
	\end{theorem}
	
	Utilizing the techniques developed$,$ we discuss examples of monoids in $\mathfrak{M_n}$. In Section 5$,$ we show the existence of unimodular elements of projective modules over Segre extension $S_{mn}(R)$ of $R$ and as it's application$,$ investigate the Serre dimension of Rees algebras $R[\mathcal{I}t]$. For $m, n \in \mathbb{Z}_{>0},$ define $Y$ to be the $m \times n$ matrix of indeterminates $y_{ij}$ for $1 \leq i \leq m$ and $1 \leq j \leq n$. Let $I$ be the ideal of the polynomial algebra $R[y_{ij}~|~1 \leq i \leq m, 1 \leq j \leq n] = R[F]$ generated by the binomial relations obtained from the $2 \times 2$ minor of $Y$. Then the Segre extension $S_{mn}(R)$ of $R$ over $mn$ variables is defined as $ S_{mn}(R) = R[F]/I$. From (\cite{Swan-MR1144038}$,$ Lemma 12.11)$,$ $S_{22}(R)$ is a monoid algebra over $R$. If $R$ is a field$,$ then Lindel (\cite{Lindel-MR1322406}$,$ Example 1.9) proved that projective $S_{22}(R)$-modules are free. This was later extended by Swan (\cite{Swan-MR1144038}$,$ Theorem 1.5) for any Dedekind domain $R$. Krishna-Sarwar (\cite{Krishna-Parvez-MR3995720}$,$ Theorem 1.3) proved that if $d=1$ and $\mathbb{Q} \subset R,$ then $S$-$dim(S_{22}(R)) \leq 1$. In a general $m \times n$ setup$,$ we prove that the monoid corresponding to $S_{mn}(R)$ is a member of $\mathfrak{M}_{k(m,n)},$ where $k(m,n) = \Big[\frac{m+n-1}{min\{m,n\}}\Big]$ (see Theorem \ref{mmt3}): 
	
	\begin{theorem}\label{mt3}
		Let $R$ be a ring of dimension $d$ and $A= S_{mn}(R)$ be the Segre extension of $R$ over mn variables. Let $k(m,n) = \Big[\frac{m+n-1}{min\{m, n\}}\Big]$. Then there exists a monoid $M \in \mathfrak{M}_{k(m,n)}$ such that $A \simeq R[M]$. As a consequence$,$ $S$-$dim(A) \leq dim(A) - k(m,n) = d + m + n - 1 - k(m,n)$. In addition$,$ if $N \in PS(M)$ is a seminormal monoid$,$ then $N \in \mathfrak{M}_{max\{m,n\}} \subset \mathfrak{M}_{k(m,n)}$.
	\end{theorem} 
	
	The central theme in the theorems above is to check the invariance of the bound of Serre dimension under monoid extensions of the ring $R$. We next attempt to improve upon this result for a certain class of rings $R$. In (\cite{BLR-MR796196}$,$ Theorem 5.2)$,$ for a normal $d$-dimensional ring $R$ and a projective $R[T]$-module of $rank \geq d,$ the authors proved the map $Um(P) \rightarrow Um(P/TP)$ to be surjective$,$ if $Um(P_f) \neq \emptyset$ for some $f \in R[T]$ monic in $T$. Taking cue from them$,$ we prove the following (see Theorem \ref{mmt4}):

	\begin{theorem}\label{mt4}
		Let R be a normal ring of dimension $d,$ $M \in \mathfrak{M}_n$ a normal $\phi$-simplicial monoid of rank $r > 0$ and $A=R[M]$. Let $P$ be a projective $A$-module of rank $ dim(A) - n$ and $J=J(R,P)$ be the Quillen ideal of $P$. Assume
		\begin{enumerate}
			\item  $Um(P_f) \neq \emptyset$ for some $f \in R[M]$ monic in $t_1$;
			\item When $n>1,$ $M \in \mathfrak{M}_n$ is such that the automorphism $\widetilde{\eta}$ obtained has the form $\widetilde{\eta}(t_i) \in  t_i + M_1$ for $i > 1$.
		\end{enumerate}
		Then the map $Um(P) \rightarrow Um(P/A_+^1P)$ is surjective.
	\end{theorem}
	
	As a corollary to the above$,$ we show that if $R$ is a normal ring$,$ $B$ a birational overring of $R[X]$ and $M$ a positive normal rank 2 monoid$,$ then $S$-$dim(B[M]) \leq dim(R)$ (Corollary \ref{e1}).
	
	
	\section{Preliminaries}
	
	\textbf{Throughout$,$ all rings are commutative noetherian with unity and projective modules are finitely generated of constant rank. 
		All monoids considered are commutative$,$ cancellative and torsion-free.}
	
	For a projective $R$-module $P$ and $p \in P,$ we denote the \textit{order ideal} of $p$ as $\mathcal{O}_P(p) = \{ \phi(p)~|~\phi \in Hom_R(P,R)\}$. If $\mathcal{O}_P(p) = R,$ then $p$ is called a \textit{unimodular element} of $P$. The set of such elements is denoted by $Um(P)$ and is our main object of study in this paper. The existence of a unimodular element $p \in P$ decomposes it to $P \simeq Rp \oplus Q,$ where $Q$ is a projective $R$-module. This is often referred to as a splitting of $P$.  
	
	The monoid algebra $R[M]$ is generated as a free $R$-module with basis as elements of the monoid $M$ and coefficients in $R$. For a finite set $T$ we denote by $\pz{}[T],$ the monoid generated by finite multiplicative $\pz{}$-combinations of elements in $T$. A monoid $M$ is defined to be 
	\begin{itemize}
		\item \textit{affine}$,$ if it is finitely generated$,$ i.e.$,$ there exists $m_1, \ldots, m_k \in M$ such that $M = \pz{}[m_1, \ldots , m_k]$;
		\item \textit{positive}$,$ if its group of units $U(M)$ is trivial;
		\item \textit{cancellative}$,$ if $zx=zy$ implies $x=y$ for $x, y, z \in M$;
		\item \textit{torsion-free}$,$ if for $x,y\in M$ and $n > 0,$ $x^n = y^n$ implies $x=y$;
		\item \textit{normal}$,$ if $x \in gp(M)$ and $x^n \in M$ for some $n >0,$ then $x \in M$;
		\item \textit{seminormal}$,$ if $x \in gp(M)$ and $x^2, $ $x^3 \in M,$ then $x \in M$.
	\end{itemize}
	Since $M$ is cancellative$,$ the torsion-freeness of $M$ is equivalent to that of its group completion $gp(M)$. The rank of $M$ is the dimension of the $\mathbb{Q}$-vector space $\mathbb{Q} \otimes gp(M)$.	
	
	Gubeladze in (\cite{Gubeladze2-MR937805}$,$ Corollary 3.2)$,$ proved the following:
	
	\begin{theorem}\label{tp2}
		Let $k$ be a field and $A = k[M]$ be the monoid algebra. Then $M$ is seminormal if and only if projective $A$-modules are free.
	\end{theorem}
	
	For a submonoid $L \subset M,$ one can define the localization of $M$ at $L,$ as the submonoid of $gp(M)$ given by $L^{-1}M = \{m-l~|~m \in M, l \in L\}$. An affine monoid $M$ of rank $r$ is \textit{$\phi$-simplicial} if there exists an embedding $M \hookrightarrow \pz{r},$ which is integral. Geometrically speaking$,$ it means we can find a hyperplane $\mathcal{H} \subset \mathbb{R}^r,$ such that $H$ intersects $\mathbb{R}_+M$ (the cone generated by $M$) in a simplex. The equivalent definition that we will use during the course of our discussion is that $M \subset \pz{}[t_1, \ldots, t_r]$ and for all $i$ there exists $p_i \in \mathbb{Z}_{>0}$ such that $t_i^{p_i} \in M$. 
	
	As is usually the case if $M \subset \pz{}[t_1, \ldots, t_r]$ is not $\phi$-simplicial$,$ we can define a new class of monoids containing $M,$ denoted by $PS(M).$ We say $N \in PS(M),$ if $M \subset N \subset \pz{}[t_1, \ldots, t_r]$ and if for an $i,$  $t_i^{p_i} \notin M$ for all $p_i \in \mathbb{Z}_{>0},$ then there exists a $s_i \in \mathbb{Z}_{>0}$ such that $t_i^{s_i} \in N$. In simpler terms we can identify elements of $PS(M)$ as a monoid obtained by adjoining elements to $M$ to make it $\phi$-simplicial. 
	
	If $M$ is a positive affine monoid of rank $r,$ then we know that $M$ can be embedded in $\pz{r} ( \simeq \pz{}[t_1, \ldots, t_r])$. For $T \subset \{1, \ldots, r\}$,  define $ \widehat{M}_T = \pz{}[t_i~|~i \notin T] \cap M $. Let $M_1 = M \cap \pz{}\bigg[ \prod\limits_{i=1}^{r} t_i^{p_i} \mid p_1 > 0, p_ i \in \pz{} \bigg]$. Given $1 \leq j \leq r,$ one can assign a positive grading to $A = R[M]$ via $t_j,$ by defining $A = \bigoplus\limits_{i\geq 0} A_i^j = A_0^j \oplus A_+^j$ where
	\begin{enumerate}
		\item  $M_i^j: = \pz{}[t_1, \ldots, \widehat{t_j}, \ldots, t_r ]t_j^i \cap M;$
		\item $A_i^j$ is the $R[M_0^j]$-module generated by $M_i^j;$
		\item $A_+^j = \bigoplus \limits_{i \geq 1}A_i^j$.
	\end{enumerate}
	Observe $A_0^j = R[M_0^j]$ is a monoid algebra where $M_0^j = \widehat{M_j}$ is a positive monoid (and seminormal$,$ if $M$ is so) of rank $r-1$.
	
	Unless specified$,$ $R[M]$ is assumed to have the $\pz{r}$-grading corresponding to the lexicographic order $t_1 < t_2 < \ldots < t_r$. One may note that in this case the zeroth homogeneous component $A_0 = R,$ and the irrelevant ideal $A_+$ is generated by the generators of the monoid. Let $A = R[t_1, \ldots, t_r]$ and $f\in A$. Then corresponding to this order we define $H(f)$ to be the highest degree component of $f$ and by $L(f),$ the coefficient of $H(f)$ in $R$. We say $f \in A$ is a \textit{quasi-monic} if $L(f) \in R^{\times} = U(R)$. 
	
	The following result of Lindel (\cite{Lindel-MR1322406}$,$ Proposition 1.8) is vital to our discussion:
	
	\begin{proposition}\label{pp2}
		Let $A = \bigoplus\limits_{i \geq 0} A_i$ be a positively graded algebra and $P$ be a projective $A$-module. Let $P_s$ be free for some $s \in A_0$. If there exists $p \in P$ such that the extension $A_0/ O_P(p) \cap A_0 \rightarrow A/O_P(p)$ is integral and $O_P(p) + sA_+ = A,$ then there exists $q \in Um(P)$ such that $q \equiv p$ modulo $sA_+P$.
	\end{proposition}
	
	We tweak Lemma 3.1 of \cite{Bhatwadekar-birational-MR978296} and write
	
	\begin{lemma}\label{pl1}
		Let $R \subset S$ be a finite extension of reduced rings. Assume $(R/C)_{red} = (S/C)_{red},$ where $C=Ann_R(S/R)$. Assume $P$ is a projctive $R$-module of $rank\geq 2$. If $Um(P \otimes_R S) \neq \emptyset,$ then $Um(P) \neq \emptyset$.   
	\end{lemma}
	
	Let $sn(R)$ and $sn(M)$ denote the seminormalization of the respective entities. We invoke the following Theorem (\cite{Gubeladze1-MR2508056}$,$ Corollary 4.76):
	
	\begin{theorem}\label{pt1}
		Let $R$ be a reduced ring and M a monoid. Then $sn(R[M]) \simeq sn(R)[sn(R[M])]$.
	\end{theorem}
	
	Let $R$ be a reduced ring and $P$ a projective $R[M]$-module of $rank  \geq 2.$ Then using the above two results one can show that if $Um(P \otimes sn(R)[sn(M)]) \neq \emptyset,$ then $Um(P) \neq \emptyset$. This is of immense utility in investigating unimodular elements of projective $R[M]$-module when $M$ is not seminormal.
	
	
	\section{The top rank case}
	This section is dedicated to proving the top rank case for monoid algebras $R[M]$. The following lemma is a crucial step towards the proof of Theorem \ref{mt1}.
	\begin{lemma}\label{l1}
		Let R be a ring of dimension d and $A = R[X_1, 
		\ldots, X_m][M] ,$ where $M = \pz{}[W_1, \ldots, W_l]$ is a positive monoid of rank $r \geq s$. Let $I \subset A$ be an ideal of height $> d + r - s$. If 
		$S =  \{ W_{j_1}, \ldots, W_{j_{s}}\} \subset \{W_1, \ldots, W_l \}$ is an algebraically independent subset of $A,$ then $I \cap R[S, X_1, \ldots, X_m]$ contains a quasi-monic.
	\end{lemma}
	\begin{proof}
		Let $W'=\{Z_1, \ldots, Z_l, X_1, \ldots, X_m \}$ be a set of variables over $R$. Consider the following composition of maps:
		\begin{linenomath*}
			\[R[Z_{j_1}, \ldots, Z_{j_{s}}, X_1, \ldots, X_m ] \xhookrightarrow{  \hspace{0.2cm} i \hspace{0.2cm} } R[W'] \xtwoheadrightarrow{\beta} A,\]
		\end{linenomath*}
		where $\beta\big|_{R} = Id_{R}$ and $\beta(Z_j) = W_j$ for all $1 \leq j \leq l$ and $i$ is the natural inclusion. Then $ht(\beta^{-1}(I)) > (d + r - s) + ht(\ker(\beta)) \geq d + l - s$. This in turn implies $ht(\beta^{-1}(I) \cap R[ Z_{j_1}, \ldots Z_{j_s}, X_{1}, \ldots, X_{m}]) > d$. The result follows from (\cite{Lam-MR2235330}$,$ Ch. III$,$ Lemma 3.2)$,$ which implies the existence of $f \in I \cap R[S][X_1, \ldots, X_m]$ with $L(f) \in R^{\times} = U(R)$. 
	\end{proof}
	
	\begin{proposition}\label{p2}
		Let R be a ring of dimension $d > 0,$ $M \subset \pz{r}$ a positive monoid of rank r and $A=R[M]$. Let $P$ be a projective $A$-module of rank $>d$. Consider the following patching diagram:
		\begin{figure}[H]
			\tikzset{column sep=small, row sep=small, ampersand replacement=\&}
			\centering
			\begin{floatrow}
				\ffigbox{\begin{tikzcd}
						B \arrow[rr,""] \arrow[dd,swap,""] \&\&
						B_1 \arrow[dd,"\pi_1"] \\
						\&  \\
						B_2 \arrow[rr, "\pi_2"] \&\& B' .
				\end{tikzcd}}{}
			\end{floatrow}
		\end{figure} 
		\begin{enumerate}
			\item Let $rank(P) > 2$ and $s \in R$ be such that $B = A, B_1 = A_s, B_2 = A_{1+sR}, B' = A_{s(1+sR)}$ or
			\item Let $s \in R$ be a non-zerodivisor such that $B = A/sA_+^i, B_1 = A/sA, B_2 = A/A_+^i$ and $B' = A/(s, A_+^i)$.
		\end{enumerate}
		If $Um(P \otimes_A B_j) \neq \emptyset$ for $j=1,2,$ then $Um(P \otimes_A B) \neq \emptyset$.
	\end{proposition}
	\begin{proof}
		Since $nil(R) \subset nil(A),$ we can assume $R$ is a reduced ring. As $rank(P) \geq 2,$ we may further assume $R[M]$ to be seminormal using Lemma \ref{pl1}. By Theorem \ref{pt1} we have $M$ is seminormal. Let $P$ be a projective $A$-module satisfying the conditions of the hypothesis. Choose $u \in Um(P \otimes_A B_1)$ and $v \in Um(P \otimes_A B_2)$. Consider the decomposition arising from $\pi_1(u) \in Um(P \otimes_A B')$ as $P \otimes_A B' = B' \oplus Q,$ where $Q$ is projective $B'$-module of rank $> d -1$. If $rank(Q) \geq max\{2,d\},$ then by (\cite{Maria-MKK1}$,$ Theorem 3.4) there exists a $\widetilde\sigma \in E(P \otimes B')$ such that $\widetilde\sigma(\pi_2(v)) = \pi_1(u)$. The corresponding diagrams will be
		
		\begin{figure}[H]
			\tikzset{column sep=small, row sep=small, ampersand replacement=\&}
			\centering
			\begin{floatrow}
				\ffigbox{\begin{tikzcd}
						A \arrow[rr,""] \arrow[dd,swap,""] \&\&
						A_s \simeq R_s[M] \arrow[dd,"\pi_1"] \\
						\&  \\
						A_{1+sR} \simeq R_{1+sR}[M]  \arrow[rr, "\pi_2"] \&\& A_{s(1+sR)}
				\end{tikzcd}}{}
				\ffigbox{\begin{tikzcd}
						A/sA_+^{i} \arrow[rr,""] \arrow[dd,swap,""] \&\&
						A/sA \simeq (R/s)[M] \arrow[dd,"\pi_1"] \\
						\&  \\
						A/A_+^{i} \simeq R[M_0^i] \arrow[rr, "\pi_2"] \&\& A/(s, A_+^i).
				\end{tikzcd}}{}
			\end{floatrow}
		\end{figure}
		
		(1) By (\cite{Quillen-MR427303}$,$ Lemma 1 and Theorem 1)$,$ there exists a decomposition $\widetilde\sigma = \alpha_{1+sR}\circ \beta_s,$ where $\alpha \in Aut(P_s)$ and $\beta \in Aut(P_{1+sR})$. As $\beta_s(v_s)  = \alpha_{1+sR}^{-1}(u_{1+sR}) $. Patch $\beta(v)$ and $\alpha^{-1}(u)$ to get $p \in Um(P)$.
		
		(2) Lift $\widetilde\sigma$ to $\sigma \in E(P/A_+^iP),$ and patch $\sigma(v)$ and $u,$ to get $\bar{p} \in Um(P/sA_+^iP)$. If $d = 1$ and $rank(Q) = 1,$ then follow the approach as in Case 2 of (\cite{MKK&HPS-MR3647151}$,$ Theorem 3.4) to get $ Um(P/sA_+^iP) \neq \emptyset$.
	\end{proof}
	
	The following lemma restructures the generators of $M$:
	
	\begin{lemma}\label{l2}
		Let $M \subset \nz{}[t_1, \ldots, t_r]$ be an affine monoid of rank $r$. Then there exists generators $\{W_1, \ldots, W_l \} \subset M$ such that
		\begin{enumerate}
			\item $\{W_1, \ldots, W_{2r'}\} = gen(U(M)),$  where $r' = rank(U(M))$
			\item $W_i = W_{i+r'}^{-1}$ for $1 \leq i \leq r'$
			\item $V = \pz{}[W_{2r'+1}, \ldots, W_l]$ is a positive monoid such that $V \subset M \smallsetminus U(M)$.
		\end{enumerate}
	\end{lemma}
	\begin{proof}
		Let $W' = \{W'_1, \ldots, W'_{l'}\} \subset \nz{}[t_1, \ldots, t_r] $ be a minimal set of generators of $M$. Then for some $k \leq l',$ $\{W'_1, \ldots, W'_k\}$ acts as a generating set of $U(M)$. As $U(M)$ is a finitely generated torsion-free group$,$ it is free and generated by $r'$ elements such that $r \geq r' = rank(U(M))$.  Let $U(M) \stackrel{\theta}{\simeq} \pz{}[x_1^{{\pm}1}, \ldots, x_{r'}^{\pm 1}]$ and $ \theta^{-1}x_i = W_i,$ where $W_i \in \pz{}[W'_1, \ldots, W'_k]$. Then $ \pz{}[W_1^{\pm{1}}, \ldots, W_{r'}^{\pm 1}] = \pz{}[W'_1, \ldots, W'_k] = U(M)$. Then $W = \{W_1^{\pm{1}}, \ldots, W_{r'}^{\pm 1}, W'_{k+1}, \ldots, W'_l\}$ is a generating set of $M$ with the desired properties.
	\end{proof}	
	
	The advantage of such a decomposition is that it gives a useful way to cleave the units and gives sufficient data about its components. We henceforth refer to $V$ as a \textit{positive component}$,$ $p(M)$ of $M$. A nontrivial consequence of the above lemma would be that if $M$ is not free$,$ then $rank(U(M)) < rank(M)$. Now we are in a position to prove Theorem \ref{mt1}.
	
	\begin{theorem}\label{mmt1}
		Let R be a ring of dimension d and $M$ be a monoid of rank $r \geq 1$. Then $S$-$dim(R[M]) \leq max\{ 1, dim(R[M])-1 \} = max \{1, d + r - 1 \}$.
	\end{theorem}

	\begin{proof}
		We may assume $R$ to be a reduced ring. Let $M$ be affine and $A=R[M]$. Assume $P$ to be a projective $A$-module of rank $> max \{1, dim(A)-1 \} = max\{1, d+r-1\}$. As $rank(P) \geq 2,$ we may further assume $R[M]$ to be seminormal using Lemma \ref{pl1}. By Theorem \ref{pt1}, $M$ is seminormal. We will induct on the dimension of the base ring $R$. If $d=0,$ then $R$ is a finite product of fields and by Theorem \ref{tp2}$,$ $P$ is free. If $r=1,$ then as $M$ is seminormal$,$ we have $M = \pz{}$ or $M = \nz{}$ and we are done by \cite{BLR-MR796196}.
		
		Let $d>0$ and $r >1$.  Then $rank(P) > \max\{1, d+r-1\} >2$. Denote by $S$ the set of non-zerodivisors of $R$. Then by the $d=0$ case on $S^{-1}A,$ there exists a non-zerodivisor $s \in R$ such that $P_s$ is free. By induction on $d,$ $Um(P/sP) \neq \emptyset$. Thus there exists a $p \in P$ such that $(p,s) \in Um(P \oplus A),$ the image of $p$ is unimodular in $P/sP$ and $O_P(p) + sA = A$. By (\cite{Lindel-MR1322406}$,$ Corollary 1.3) we may assume $ht(O_P(p)) = rank(P) \geq d + r $.
		
		By Lemma \ref{l2} we may assume existence of $W = \{W_1^{\pm{1}}, \ldots, W_{r'}^{\pm{1}}, W_{2r'+1}, \ldots, W_{2r'+l}\} \subset M \subset \nz{}[t_1, \ldots, t_r]$ such that $M = \pz{}[W]$. Employ the following composition of maps:
		\begin{linenomath*}
			\[R[X_i^{\pm{1}}] \xhookrightarrow{  \hspace{0.2cm} i \hspace{0.2cm} } R[ X_{1}^{\pm{1}}, \ldots, X_{r'}^{\pm{1}}, X_{2r'+1}, \ldots, X_{2r'+l}] \xtwoheadrightarrow{\beta} A,\]
			\[R[X_j] \xhookrightarrow{  \hspace{0.2cm} j \hspace{0.2cm} } R[ X_{1}^{\pm{1}}, \ldots, X_{r'}^{\pm{1}}, X_{2r'+1}, \ldots, X_{2r'+l}] \xtwoheadrightarrow{\beta} A,\]
		\end{linenomath*}
		where $1 \leq i \leq r',$ $2r'+1 < j \leq 2r'+ l,$ $\beta\big|_R = Id_R$  and $\beta(X_k) = W_k$ for all $k$. As height of both the ideals $O_P(p) \cap R[X_i^{\pm{1}}]$ and $O_P(p) \cap R[X_j]$ exceed $d,$ by Lemma \ref{l1} and (\cite{Mandal1-MR728138}$,$ Lemma 2.3)$,$ there exists monics $f_j \in O_P(p) \cap R[W_j]$ and special monics $f_i \in O_P(p) \cap R[W_i],$ with coefficients in $R$. This implies the extension $R/O_P(p) \cap R \hookrightarrow A/O_P(p)$  is integral. As $O_P(p) + sA = A,$ hence $(O_P(p) \cap R) + sR = R$. This in turn means $Um(P_{1+sR}) \neq \emptyset$ and by Proposition $\ref{p2}$ we can conclude the proof.
		
		If $M$ is not affine$,$ then $M$ can be written as a filtered union of affine monoids $M_{i},$ where $i \in \mathcal{I}$. As seminormalization of affine monoids is again affine$,$ we can write $M$ as the filtered union of affine seminormal monoids. Thus a projective $R[M]$-module is extended from a projective $R[M_i]$-module for some $i \in \mathcal{I}$ and the rest follows.
	\end{proof}
	
	If $M$ is a positive seminormal monoid of rank $r,$ then by \cite{BLR-MR796196}$,$ $S$-$dim(R[M \oplus \nz{n}]) \leq dim(R[M]) = d +r$. One way of looking at this would be that the units of monoids play no role in the Serre dimension in the above case. A natural question in a general setup would be whether $S$-$dim(R[M]) \leq d + r - rank(U(M))$? The corollary below gives a partial answer to this query:
	
	\begin{corollary}\label{c4}
		Let $R$ be a ring of dimension $d$ and $M$ be a normal monoid of rank $r$. Then $S$-$dim(R[M]) \leq d + r - rank(U(M))$. In particular$,$ if 
		\begin{enumerate}
			\item $rank(U(M))=r-1,$ then $S$-$dim(R[M]) \leq d$;
			\item $r = 2,$ then $S$-$dim(R[M]) \leq d$.
		\end{enumerate}
	\end{corollary}
	
	\begin{proof}
		Let $M$ be affine and $rank(U(M)) = r'$. Then by (\cite{Gubeladze1-MR2508056}$,$ Proposition 2.26) we get that $M \simeq U(M) \oplus M'$ and by \cite{BLR-MR796196}$,$  $S$-$dim(R[M]) = S$-$dim(R[M' \oplus U(M)]) \leq d + r - r'$. For (1) observe that $M'$ will be a free positive monoid of rank 1 and use \cite{BLR-MR796196} to get the indicated. If $r=2$ and $M$ is not positive ($r' > 0$) then the assertion follows from (1). When $M$ is positive$,$ invoke (\cite{MKK&HPS-MR3647151}$,$ Corollary 3.6) to prove the required. If $M$ is not affine$,$ then by (\cite{Gubeladze1-MR2508056}$,$ Proposition 2.22) we may write $M = \lim\limits_{\longrightarrow} M_i,$ where $M_i$'s are affine normal monoids. Then a projective $R[M]$-module is extended from a projective $R[M_i]$-module and thus the assertions follow.
	\end{proof}
	
	
	\section{Serre dimension of monoid algebras corresponding to $\mathfrak{M}_n$ - Lower rank case}

	Define $\mathfrak{M}_1$ to be the class of affine positive monoids. Let $M \in \mathfrak{M}_1$ be a submonoid of $\pz{}[t_1, \ldots, t_r]$ of rank $r$. Fix $W = \{W_1, \ldots, W_l\} \subset M$ such that $M = \pz{}[W],$ where $W_i$'s are monomials in $\pz{}[t_1, \ldots, t_r]$. Define $gen(M)_1 = W \cap (M \setminus \widehat{M}_1) = \{U_{1}, \ldots, U_{g_1}\},$ which is a subset of generators of $M$ having some positive power of $t_1$.

	For $n \geq 2,$ define class of monoids $\mathfrak{M}_n \subset \mathfrak{M}_1$ as $M \in \mathfrak{M}_n$ if
	\begin{enumerate}
		\item $n \leq r = rank(M)$;
		\item For each $U_{i} \in gen(M)_1,$ there exists algebraically independent set $S_{i} = \{W_{i_1}, \ldots, W_{i_n}\} \subset W$ such that if $f_i \in R[S_i]$ is quasi-monic$,$ then there exists $\eta \in Aut(R[M]),$ which is the restriction of $\widetilde{\eta} \in Aut_{R[t_1]}(R[t_1, \ldots, t_{r}]),$ such that $\eta(f_i)$ is monic in $U_i$ with coefficients in $A_0^1$ for each $1 \leq i \leq g_1$;
		\item $\widehat{M_{1}} \in {\mathfrak{M}_{n-1}}$.
	\end{enumerate}
	
	For the sake of convenience if $M \in \mathfrak{M}_n,$ we will denote by $(U_j, S_j)_{1 \leq j \leq g_1}$ it's relevant information at the $n$'th level. The second property guarantees the existence of monic polynomials in all of the elements of $gen(M)_1$ in ideals of large enough height. The third ensures a smooth inductive process. These two steps$,$ in combination with a Milnor patching diagram$,$ lead to the required result. One may note that the classes  of monoids $\{\mathfrak{M}_n\}_{n \geq 1}$ defined above, form a descending chain since $\mathfrak{M}_i \subseteq \mathfrak{M}_j$ if $i \geq j$.

	Let $M$ be $\phi$-simplicial and therefore for all $i,$ $t_i^{p_i} \in M$ for some $p_i \geq 1$. Then $t_1^{p_1} \in gen(M)_1,$ irrespective of one's choice for generators for $M$. To decrease the Serre dimension when $M$ is $\phi$-simplicial monoid, we don't need to  find monics in all the elements of $gen(M)_1$ but just $t_1^{p_1} \in gen(M)_1$. Thus if $M$ is $\phi$-simplicial then condition (2) can be relaxed to
	
	\begin{enumerate}
		\item[(2$'$)] If $f \in R[t_1, \ldots, t_n] \cap R[M]$ is a quasi-monic$,$ then there exists an $\widetilde{\eta} \in Aut_{R[t_1]}(R[t_1, \ldots, t_r]),$ such that its restriction  $\eta \in Aut(R[M])$ and $\eta(f)$ is monic in $t_1$ with coefficients in $A_0^1$.
	\end{enumerate}
	
	\begin{remark}\label{r1}
		If $M$ is a $\phi$-simplicial normal monoid$,$ then by (\cite{Gubeladze-umodrow1-MR1161570}$,$ Lemma 6.3) we can relax the definition of $\mathfrak{M}_n$ and demand that under the automorphism we get a monic with coefficients in $R[M]$ instead of $R[M_0^1]$.
	\end{remark}
	
	If $M$ is $\phi$-simplicial, then $M \in \mathfrak{M}_n$ means $M$ satisfies conditions $(1),$ $(2')$ and $(3)$. The index chosen in the above definitions hold no relevance as we can permute the variables. For projective modules $P$ of rank $d < rank P < d+r,$ we will now prove Theorem \ref{mt2}:
	
	\begin{theorem}\label{mmt2}
		Let R be a ring of dimension d and $A=R[M],$ where $M \in \mathfrak{M}_n$ is a seminormal monoid of rank $r \geq 1$.  Assume $P$ to be a projective $A$-module of $rank > dim(A) - n = d +r-n$. Then 
		\begin{enumerate}
			\item the map $Um(P) \rightarrow Um(P/A_+^1P)$ is surjective and
			\item $S$-$dim(A) \leq dim(A) - n = d + r - n$.
		\end{enumerate}
		In particular$,$ if $M \in \mathfrak{M}_r,$ then $S$-$dim(A) \leq d $.
	\end{theorem}
	
	\begin{proof}
		One can assume that $R$ is reduced. We will proceed by induction on $d$. If $d=0,$ then $R$ is a product of fields and by Theorem \ref{tp2}$,$ both $P$ and $Um(P/A_+^1P)$ are free. As $R[M] \rightarrow R[M_0^1]$ is a retraction$,$ the surjection follows. Assume $d>0$. Let $S$ be the set of non-zerodivisors of $R$. Then $dim(S^{-1}R)=0,$ and by $d=0$ case we can find an $s \in S$ such that $P_s$ is free. 
		
		Let $p_1 \in Um(P/A_+^1P)$. By the inductive process$,$ $Um(P/sP) \neq \emptyset$ and thus Proposition \ref{p2} gives $\bar{p} \in Um(P/sA_+^1P)$  such that $p \equiv p_1$ modulo $A_+^1P,$ where '$-$' denotes reduction modulo $sA_+^1$. As $O_P(p) + sA_+^1 = A,$ choose $a \in A_+^1$ such that $1 + sa \in O_P(p).$ By (\cite{Lindel-MR1322406}$,$ Lemma 1.2 and Corollary 1.3) there exists $q \in P$ such that $ht(O_P(p + saq)) = rank(P) \geq d + r - n + 1 $. As $p$ and $p +saq$ have the same image modulo $sA_+^1P,$ we can assume $ht(O_P(p)) \geq d+ r - n + 1$.
		
		Let  $(U_j, S_j)_j$ be the relevant information of $M$ at the $n$'th level. As per definition$,$ corresponding to each $U_j \in gen(M)_1$ there exist algebraically independent subset $S_j = \{W_{j_1}, \ldots, W_{j_n}\}$ of $M.$ Consider the composition of maps
		\begin{linenomath*}
			\[R[X_{j_1}, \ldots, X_{j_n}] \xhookrightarrow{  \hspace{0.2cm} i_j  \hspace{0.2cm} } R[X_1, \ldots, X_l] \xtwoheadrightarrow{\beta} R[M].\]
		\end{linenomath*}
		By Lemma \ref{l1} there exists quasi-monic $f_j \in O_P(p) \cap R[S_j]$'s for all $j$. By definition of $\mathfrak{M}_n,$ there exists a common $\widetilde{\eta} \in Aut_{R[t_1]}(R[t_1, \ldots, t_r])$ such that its restriction $\eta \in Aut(A)$ and $\eta(f_j)\in \eta(O_P(p))$ is monic in $U_j$ with coefficients in $\eta(A_0^1)$ for all $j$. As $\widetilde{\eta}$ fixes $t_1,$ we have $\eta(O_P(p)) + s\eta(A_+^1) = A$. Replacing $A$ by $\eta(A)$ we can assume $O_P(p)$ contains monics in $U_j$ with coefficients in $A_0^1$ for all $j$ and $O_P(p) + sA_+^1 = A$. Hence the extension $A_0^1/O_P(p) \cap A_0^1 \hookrightarrow A/O_P(p)$ is integral. Using Proposition \ref{pp2} we obtain a $p_0 \in Um(P)$ such that $\bar{p_0} = \bar{p}$. Note that $p_0 \equiv p \equiv p_1$ modulo $A_+^1$. Thus the map $Um(P) \rightarrow Um(P/A_+^1P)$ is surjective. The second conclusion follows using process similar to above with an added induction on $n,$ where $n=1$ case follows through by Theorem \ref{mmt1}. 
	\end{proof}
	
	Given a projective $A$-module $P,$ we denote by $\mu(P)$ the number of minimal generators of $P$. 
	
	\begin{corollary}
		Let $R$ be a ring of dimension d and $A=R[M],$ where $M$ is a seminormal monoid of rank r. Let $P$ be a projective $A$-module. Then $\mu(P) \leq rank(P) + d + r - 1$. If $M \in \mathfrak{M}_n,$
		then $\mu(P) \leq rank(P) + d + r - n$.
	\end{corollary}
	
	\begin{proof}
		Let if possible $\mu(P) = m > rank(P) + d + r - n.$ Consider the natural surjection $\phi: A^m \rightarrow P,$ where $Q = ker(\phi)$. By Theorem \ref{mmt2}$,$ there exists $q \in Um(Q)$. As $q \in Q = ker(\phi),$ $\phi$ restricts to a surjection $\bar{\phi}: A^m/qA \rightarrow P$. By (\cite{Maria-MKK1}$,$ Theorem 3.4) $A^{m}/qA \simeq A^{m-1},$ which implies $P$ is generated by $m-1$ elements$,$ a contradiction. 
	\end{proof}
	
	\begin{corollary}\label{c3} 
		Let $M \in \mathfrak{M}_n$ be a $\phi$-simplicial monoid of rank $r$. Then $M \oplus \pz{m} \in \mathfrak{M}_{n+m}$.
	\end{corollary}
	
	\begin{proof}
		Let $M' = M \oplus \pz{m}$ and $m>0$. Let $Z_i$ represent the variables in the $\pz{m}$ direct summand of $M'$. Let $f \in R[t_1, \ldots, t_n, Z_1, \ldots, Z_m] \cap R[M']$ be a quasi-monic. Choose $c \in \pz{}$ large enough so that the $R[M]$-automorphism of $R[M']$ given by
		\begin{linenomath*}
			\[ Z_i \xmapsto{ \hspace{0.2cm} \theta \hspace{0.2cm}} Z_i +  t_{n}^c,\]
		\end{linenomath*}
		for all $i,$ is such that the coefficient of highest order component of $\theta(f)$ is given by $\displaystyle u\hspace{-0.2cm}\prod_{1 \leq i \leq n}\hspace{-0.2cm}t_{i}^{c_{i}}$ and $u \in R^{\times}$. Then $\theta(f)$ is a quasi-monic in $R'[t_1, \dots, t_n] \cap R'[M],$ where $R'=R[\pz{m}]$. By definition of $\mathfrak{M}_n,$ there exists $\widetilde{\eta} \in Aut_{R'[t_1]}(R'[t_1, \ldots, t_r])$ such that its restriction $\eta \in Aut(R'[M]) = Aut(R[M'])$ and $\eta(\theta(f))$ is monic in $t_1$ with coefficients in $(R'[M])_0^1 = R[\pz{m}][M_0^1]$. As $\widetilde{\eta} \circ \theta \in Aut_{R[t_1]}(R[t_1, \ldots, t_r][\pz{m}]),$ we have our result.
	\end{proof}
	
	If $M$ is $\phi$-simplicial$,$ we can assume the much simpler condition $(2')$ instead of $(2)$. The above approach can be replicated subject to minor changes and one can arrive at the same conclusion.
	
	\begin{theorem}\label{t6}
		Let R be a ring of dimension d and $A=R[M],$ where $M \in \mathfrak{M}_n$ is a $\phi$-simplicial seminormal monoid of rank $r \geq 1$.  Assume $P$ to be a projective $A$-module of $rank > dim(A) - n = d +r-n$. Then 
		\begin{enumerate}
			\item the map $Um(P) \rightarrow Um(P/A_+^1P)$ is surjective and
			\item $S$-$dim(A) \leq dim(A) - n = d + r - n$.
		\end{enumerate}
		In particular$,$ if $M \in \mathfrak{M}_r,$ then $S$-$dim(A) \leq d $.
	\end{theorem}
	
	By $M_*$ we mean the monoid $ (int(\mathbb{R}_+(M)) \cap M)\cup \{0\}$. For the definitions of quasi-normal and quasi-truncated monoids refer to  (\cite{Gubeladze-umodrow1-MR1161570}$,$ Section 5 and 6).
	\begin{theorem}\label{t1}
		Let $M$ be a quasi-truncated quasi-normal monoid of rank $ \geq 2$. Then $M \in \mathfrak{M}_2$. As a consequence$,$ if $M$ is $\phi$-simplicial seminormal of rank $\geq 2,$ then $S$-$dim(R[M_*]) \leq d + rank(M) - 2$. 
	\end{theorem} 
	
	\begin{proof}
		Let $M$ be a quasi-truncated quasi-normal monoid of rank $r \geq 2$ and $f \in R[t_1, t_2] \cap R[M],$ be a quasi-monic. Then by (\cite{Gubeladze-umodrow1-MR1161570}$,$ Lemma 6.7) there exists an $R[t_1]$-automorphism $\eta$ of $R[M]$ given by $\eta(t_2) = t_2 + t_1^{c}$ where $c \in \mathbb{Z}_{>0}$ and thus $\eta(f)$ is monic in $t_1$. As $M$ is quasi-normal$,$ by Remark \ref{r1} $M \in \mathfrak{M}_2$. For the second part$,$ as $M$ is seminormal by (\cite{Gubeladze1-MR2508056}$,$ Proposition 2.40) we have $M_* = (n(M))_*,$ where $n(M)$ denotes normalization of $M$. Invoke  (\cite{Gubeladze-umodrow1-MR1161570}$,$ Theorem 3.1) to obtain quasi-truncated normal monoids $Q_i$'s such that $n(M)_*$ is the filtered union of $Q_i$'s. This concludes the proof. 
	\end{proof}

	In \cite{Gubeladze-MR3853049}$,$ Gubeladze introduced the concept of tilted $R$-subalgebras of $R[t_1, \ldots, t_r]$. We repurpose this concept from the point of view of the variables and define $t_i$ to be a \textit{strongly tilted variable} of the monoid $M \subset \pz{}[t_1, \ldots, t_r] = F,$ if there exists $c \in \pz{}$ such that $t_i^sF \subset M$ for all $s \geq c$. If $M$ is affine$,$ then we say $t_i$ is a \textit{tilted variable of M} if for all $j \neq i,$ there exists a $c_j \in \pz{}$ such that $t_i^{c_j}t_j \in M$ and $t_i \in M$. It follows from definition that if $t_i$ is a tilted variable of $M \subset F,$ then any monoid $N$ such that $M \subset N \subset F$ has $t_i$ as its tilted variable. As a straightforward example of such monoid$,$ consider $M = \pz{}[t_1, t_1t_2] \subset \pz{}[t_1, t_2],$ where $t_1$ is a tilted variable of $M,$ often rewritten as $R[M]$ is $t_1$-tilted. The polynomial algebra $R[t_1, \ldots, t_r]$ is tilted in all it's variables. The existence of a tilted variable leads to our required type of automorphism: 
	
	\begin{lemma}\label{l3}
		Let R be a ring  and $A=R[M],$ where $M$ is an affine positive monoid of rank $r$. Let $t_1$ be a tilted variable of $A$. If $f \in A$ is a quasi-monic$,$ then there exists an $\eta \in Aut(A)$ such that $\eta(f)$ is monic in $t_1$.
	\end{lemma}
	
	\begin{proof} 
		Let $F=\pz{}[t_1, \ldots, t_r]$ and $f = f_{0} + f_{1} + \ldots + f_{l} \in A \subset R[F]$ be a quasi-monic of degree $l  \geq 1$. Let $f_{l} = u\prod\limits_{k=1}^{r} {t_k}^{p_{k}}$ for some $u \in R^{\times}$ and $p_{k} \in \pz{}$ for $1 \leq k \leq r$. As $A$ is $t_1$-tilted$,$ there exists a $c_j \in \mathbb{Z}_{>0}$ such that $t_1^{c_j}t_j \in M$ for all $j \neq 1$ and $t_1 \in M$. 
		
		We wish to define an $R[t_1]$-automorphism of $A$ of the form $t_k \rightarrow t_k + t_1^{q}$
		for $2 \leq k \leq r,$ which satisfies the required. As $M$ is affine we have $A = R[m_1, \ldots, m_k]$. Let $c > max$ \{$tot$-$deg(m_i)$\}. Choose $q > max \{cc_j, l\}$. Then $\eta(f)$ is monic in $t_1$ and $\eta(m_i) \in A$ for all $i,$ and hence $\eta \in  Aut(A)$.
	\end{proof} 
	
	The above lemma holds for any (not necessarily affine) $R$-algebra $A \subset R[t_1, t_2, \ldots, t_r],$ which has a strongly tilted variable and proves that $\eta(f)-f \in A$ for all $f \in R[t_1, \ldots, t_r]$. As we are dealing with affine monoid algebras$,$ we refrain from using the said version in the interest of simplicity.
	
	\begin{example}
		Using techniques developed in the previous discussion$,$ we improve the Serre dimension of some monoid algebras: 
		\begin{enumerate}
			\item If $M=\pz{r}$ (free positive monoid  of rank $r$)$,$ then $M \in \mathfrak{M}_r$.
			\item The results corresponding to the special class of monoid $\mathcal{C}(\phi)$ introduced in \cite{MKK&HPS-MR3647151}$,$ can be subsumed into that of $\mathfrak{M}_n$. To be exact$,$ if $M \in \mathcal{C}(\phi)$ is of rank $r,$ then $M \in \mathfrak{M}_r$. Hence all normal rank 2 monoids belong to $\mathfrak{M}_2$ by (\cite{MKK&HPS-MR3647151}$,$ Corollary 3.6).
			\item Let $M \subset \pz{}[t_1, \ldots, t_r]$ be a $t_1$-tilted $\phi$-simplicial quasi-normal monoid of rank $r$. Then the enveloping normal monoid corresponding to $M$ is the free monoid $\pz{r}$. Then seminormalization of $M'$ (written as $sn(M')$) contains $t_i$ for all $i,$ therefore $sn(M') = \pz{r}$. For projective $R[M']$-modules $P$ of $rank>1,$ by Lemma \ref{pl1} and Theorem \ref{pt1}, we can conclude $Um(P) \neq \emptyset$.
			\item Let $M = \pz{}[x_1, x_2, x_1x_4, x_2^2x_4, x_3x_4]$. Then $M$ is a non $\phi$-simplicial seminormal monoid of rank 4. Then  we claim that $M \in \mathfrak{M}_2$. Let $(U_1 = x_1, U_2 = x_1x_4, S_1 = \{x_1, x_2\}, S_2 = \{x_1x_4, x_3x_4 \})$ be the relevant information of $M$. Then given quasi-monics $f_j \in R[S_j],$ by following the working of Lemma \ref{l3}$,$ there exists $p \in \mathbb{Z}_{>0}$ such that $\widetilde{\eta} \in Aut_{R[x_1, x_4]}(R[x_1, x_2, x_3,x_4])$ given by $\widetilde{\eta}(x_2) = x_2 + x_1^p, ~~~~~\widetilde{\eta}(x_3)  = x_3 + x_1^px_4^{p-1}$ which restricts to an $R[x_1]$-automorphism $\eta$ of $R[M]$ and $\eta(f_j)$ is monic in $U_j$ for $j=1,2$ with coefficients in $R[M_0^1]$. Also $\widehat{M}_1 = M_0^1 = \pz{}[x_2, x_3x_4, x_2^2x_4] \simeq \pz{3} \in \mathfrak{M}_3$.
		\end{enumerate}
	\end{example}


	\section{Applications}
	\subsection{Segre Extension}
	In the following theorem$,$ we identify the monoid corresponding to the monoid algebra $S_{mn}(R)$ and maneouver it by means of a monoid isomorphism. This identification in combination with Theorem \ref{mmt2}$,$ yields a proof of Theorem \ref{mt3}.
	
	\begin{theorem}\label{mmt3}
		Let $R$ be a ring of dimension $d$ and $A= S_{mn}(R)$ be the Segre extension of $R$ over mn variables. Let $k(m,n) = \Big[\frac{m+n-1}{min\{m, n\}}\Big]$. Then there exists a monoid $M \in \mathfrak{M}_{k(m,n)}$ such that $A \simeq R[M]$. As a consequence$,$ $S$-$dim(A) \leq dim(A) - k(m,n) = d + m + n - 1 - k(m,n)$. In addition$,$ if $N \in PS(M)$ is a seminormal monoid$,$ then $N \in \mathfrak{M}_{max\{m,n\}} \subset \mathfrak{M}_{k(m,n)}$.
	\end{theorem}  
	\begin{proof}
		It is sufficient to provide a proof for the case $m \leq n$. We will induct on $n$. If $n=1,$ then $A = S_{mn}(R)$ is a polynomial algebra and we are done by \cite{Bhat-Roy-MR727374}. Assume $n \geq 2$. From (\cite{Swan-MR1144038}$,$ Lemma 12.11) $A$ is isomorphic to the monoid ring $R[M']$ presented by generators $\{y_{ij}\}_{1 \leq i \leq m, 1 \leq j \leq n}$ where $y_{ij}y_{kl} = y_{il}y_{kj}$ for $i \neq k$ and $j \neq l$. Define $M$ as
		\begin{linenomath*}
			$$M = \pz{}[x_1, \ldots x_n, x_ix_j~|~1 \leq i \leq n \text{~and~} n+1 \leq j \leq m+n-1].$$
		\end{linenomath*}
		Using $gp(M) = \pz{}[x_1^{\pm 1}, \ldots, x_n^{\pm 1}, x_{n+1}^{\pm 1}, x_{n+2}^{\pm 1}, \ldots, x_{n + m - 1}^{\pm 1}],$ we have $rank(M) = m + n -1$. Consider the monoid homomorphism $\theta : M' \longrightarrow M,$
		where the generators of $M'$ are mapped to generators of $M$ in the given order:
		\begin{linenomath*} 
			\[ 
			\theta(y_{ij})= \left\{
			\begin{array}{ll}
			x_j & i =1, \\
			x_jx_{n+i-1}& i > 1. \\
			\end{array} 
			\right. 
			\]
		\end{linenomath*}
		Note that $\theta$ preserves the relations of $M',$ i.e.$,$ $\theta(y_{ij}y_{kj}) = \theta(y_{il}y_{kj})$  for $i \neq k$ and $j \neq l$. We claim that $\theta$ is an isomorphism. Surjectivity of $\theta$ is straightforward. For injectivity$,$ consider the group homomorphism $\phi: gp(M) \longrightarrow gp(M')$ defined by
		\begin{linenomath*}
			\[ 
			\phi(x_{j})= \left\{
			\begin{array}{ll}
			y_{1j} &  1 \leq j \leq n, \\
			y_{11}^{-1} y_{(j+1-n)1} & n < j \leq m + n - 1. \\
			\end{array} 
			\right. 
			\]
		\end{linenomath*}
		For $1 \leq  i \leq n$ and $1 \leq j \leq m-1,$ we can deduce $\phi(x_ix_{n+i}) = y_{1i}y_{11}^{-1}y_{(j+1)1} = y_{(j+1)i},$ using the relation between ${y_{ij}}$'s. Restricting $\phi$ to $M,$ we get $\phi(M)=M',$ $\theta \circ \phi\big|_M = Id_M$ and $\phi\big|_M \circ \theta = Id_{M'}$. Thus $\theta$ is an isomorphism and $dim(A) = dim(R) + rank(M) = d + m + n - 1$.
		
		\textit{Claim}: $M \in \mathfrak{M}_k,$ where $k = k(m,n) = \big[\frac{m+n-1}{m}\big]$.
		
		Observe $gen(M)_1 = \{x_1, x_1x_{n+1}, \ldots, x_1x_{n+m-1}\}$ and identify it's generators as $U_1 = x_1$ and $U_j = x_1x_{n+j-1}$ for $2 \leq j \leq m$. Choose $S_1 = \{x_1, \ldots, x_k\}$ and $S_j = \{x_{n+j-1}x_1, x_{n+j-1}x_{jk-(k-1)}, \ldots, x_{n+j-1}x_{jk-1}\}$ for all $j \geq 2$. Then given quasi-monics $f_j \in R[S_j],$ consider the following map:
		\begin{linenomath*}
			\[ 
			\widetilde{\eta}(x_i)= \left\{
			\begin{array}{ll}
			x_i + x_1^{d_1} & 2 \leq i \leq k, \\
			x_i + x_1^{d_j}x_{n+j-1}^{d_j-1} & jk-(k + j-2) < i \leq jk-(j-1) \text{~and~} j \geq 2,\\
			x_i & \text{else}.
			\end{array} 
			\right. 
			\]
		\end{linenomath*}
		where $d_j > tot$-$deg(f_j)$. Here $\widetilde{\eta}$ is an $R[x_1]$-automorphism of  $R[x_1, \ldots x_n, x_{n+1}, x_{n+2}, \ldots, x_{m+n-1}].$ On restricting to $R[M],$ one may see that $\eta = \widetilde{\eta}\big|_{R[M]} \in Aut_{R[x_1]}(R[M])$ and $\eta(f_j)$ is monic in $U_j$ with coefficients in $R[M_0^1]$. As
		\begin{linenomath*}
			\[ \widehat{M_1} = M_0^1 = \pz{}[x_2, \ldots x_n, x_ix_j~|~2 \leq i \leq n \text{~and~} n+1 \leq j \leq m+n-1] \simeq \{y_{ij} \mid 1 \leq i \leq m, 1 < j \leq n\},\]
		\end{linenomath*} 
		by induction  $R[\widehat{M_1}] \simeq S_{m(n-1)}(R)$. Note that if $k>1,$ then $m < n$. By the inductive process we have $\widehat{M_1} \in \mathfrak{M}_{k'},$ where $k-1 \leq k' = \Big[\frac{m+n-2}{m} \Big] \leq k$. Hence $M \in \mathfrak{M}_k$. As $M$ is seminormal$,$ the conclusion holds by Theorem \ref{mmt2}.

		We now prove the second part of the proof again by induction on $n$. Let $N \in PS(M)$ be a seminormal monoid. If $n=1,$ then $N =  \pz{} \in \mathfrak{M}_r,$ by \cite{Bhat-Roy-MR727374}. Assume $n > 1$. Then as $N$ is a $x_1$-tilted $\phi$-simplicial monoid$,$ we only need to find a monic in $x_1$. Let $f \in R[x_1, \ldots, x_n]$ be a quasi-monic. Then there exists large enough $p$ such that the $R[x_1, x_{n+1}, \ldots, x_{m+n-1}]$-automorphism of $R[x_1, \ldots, x_{m+n-1}]$ given by $x_i \mapsto x_i + x_1^p$ for $1 < i \leq n,$ restricts to an $R[x_1]$-automorphism $\eta$ of $R[N]$. By Lemma \ref{l3}$,$ $\eta(f)$ is monic in $x_1$ with coefficients in $R[x_2, \ldots, x_n] \subset R[\widehat{N_0^1}]$. As $\widehat{N_1} = N_0^1 \in PS(\widehat{M_1})$ is seminormal and $R[\widehat{M_1}] \simeq S_{m(n-1)}(R),$ by induction $\widehat{N_1} \in \mathfrak{M}_{n-1}$. Thus $N \in \mathfrak{M}_{n}$. Since we have assumed $m \leq n,$ we get $N \in \mathfrak{M}_{max\{m, n\}} \subset \mathfrak{M}_{k(m,n)}$.
	\end{proof}
	
	\subsection{Rees Algebra}
	
	Let $R$ be a ring and $\mathcal{I} = \{I_n\}$ a filtration of $R,$ where $I_j \subset I_{j-1}$ and $I_0 = R$. Denote by $R[\mathcal{I}t] = \bigoplus \limits_{n \geq 0} I_nt^n$ and $R[\mathcal{I}t, t^{-1}] = \bigoplus \limits_{n \in \mathbb{Z}} I_nt^n,$ the Rees algebra and extended Rees algebra of $R$ w.r.t. $\mathcal{I},$ respectively. If $\mathcal{I}_n$ is the $I$-adic filtration of R$,$ then by $R[It]$ and $R[It,t^{-1}]$ we denote the corresponding Rees algebras. These rings are also referred to as blowup algebras as $Proj(R[It])$ is the blowup of $Spec(R)$ along the subscheme defined by $I$ and have their application in the study of desingularization. From an algebraic viewpoint it was studied by Rees in \cite{Rees-MR95843}. 
	
	In \cite{Rao-Sarwar}$,$ Rao-Sarwar proved that $S$-$dim(R[It]) \leq dim(R),$ when $R$ is a domain. The following generalizes this result when $R$ and $\mathcal{I}_n$ have a certain form and follows as a straightforward corollary to Theorem \ref{mt1}.
	
	\begin{proposition}
		Let $R$ be a ring of dimension $d$ and $B = R[X_1, \ldots, X_m]$. Let $\mathcal{I} = \{I_n\}$ be a filtration of $B,$ where ${I}_n \subset B$ are ideals generated by non-constant monomials for all $n,$ and $A=B[\mathcal{I}t]$ or $B[\mathcal{I}t, t^{-1}]$. Then $S$-$dim(A) \leq max \{1, d \}$.
	\end{proposition}

	\begin{proof}
		We may assume $A$ is reduced. Since $I_n$ consists of monomials in $R[X_1, \ldots,X_m]$, $A$ is a monoid algebra, say $R[M],$ where $M$ is a positive monoid. Let $P$ be a projective $A$-module of $rank > max\{1, d\}$. As $rank(P) \geq 2,$ we may further assume $R[M]$ to be seminormal using Lemma \ref{pl1}. By Theorem \ref{pt1}, $M$ is seminormal.  The conclusion thus follows from Theorem \ref{mmt1}.
	\end{proof} 
	
	What is of interest is$,$ if and when we can prove that the corresponding monoid $M \in \mathfrak{M}_n$ for $n>1$. We discuss one such example:
	
	\begin{corollary}
		Let $B = R[X_1, \ldots, X_m]$ and $I  = (X_1, \ldots, X_m)$. If $A = R[It],$ then
		\begin{linenomath*} 
			\[ 
			S\text{-}dim(A) \leq \left\{
			\begin{array}{ll}
			dim(R) + (m+1)/2 & m \text{ odd;} \\
			dim(R) + (m/2)+1 & \text{else.}\\
			\end{array} 
			\right. 
			\]
		\end{linenomath*}
	\end{corollary}
	\begin{proof}
		Observe that $A \simeq S_{2m}(R) $. By Theorem \ref{mmt3}$,$ we have $A=R[M],$ where $M \in \mathfrak{M}_n$ and $n = \big[\frac{m+2-1}{2}\big]$. As $dim(A) = dim(R) + m + 1,$ the conclusion follows.
	\end{proof}
	
	\subsection{Monic Inversion}
	
	A ring $R$ is said to be normal if for every prime ideal $\mathfrak{p} \subset R,$ $R_{\mathfrak{p}}$ is a normal domain. Let $A = \bigoplus \limits_{i \geq 0}A_i$ be a positively graded ring and $P$ be a projective $A$-module. The Quillen ideal of $P,$ denoted by $J(A_0, P),$ is defined to be the set of elements $a \in A_0$ such that $P_a$ is extended from $(A_0)_a$. When $dim(A_0) \geq 1,$ it can be deduced from Theorem \ref{tp2} that $ht(J(A_0,P)) \geq 1$. Let $R$  be a $d$-dimensional normal ring and $P$ be a projective $R[T]$-module of rank $d$. Then by (\cite{BLR-MR796196}$,$ Theorem 5.2) the authors proved that the map $Um(P) \rightarrow Um(P/TP)$ is surjective$,$ if $Um(P_f) \neq \emptyset$ for some $f \in R[T]$ monic in $T$. Further$,$ utilizing the techniques of (\cite{Bhatwadekar-birational-MR978296}$,$ Lemma 3.2)$,$ the assumption of normality on $R$ can be relaxed. In (\cite{MKK-ZinnaMR3874658}$,$ Corollary 3.1) this was generalized to the ring $R[T_1, \ldots, T_n]$ to show surjection of $Um(P) \rightarrow Um(P/A_+^1P)$ subject to the existence of $f \in R[T_1, \ldots, T_n]$ monic in $T_1$ with $Um(P_f) \neq \emptyset$. We further generalize it to monoid algebras and give a proof for Theorem \ref{mt4} using the theorem below (\cite{BLR-MR796196}$,$ Criterion 1):
	
	\begin{theorem}\label{tp1}
		Let $A = \bigoplus \limits_{i \geq 0}A_i = A_0 \oplus A_+$ be a positively graded ring. Let $P$ be a projective $A$-module and $J = J(A_0,P)$. If $q \in P$ is such that $q_{1+A_+} \in Um(A_{1+A_+})$ and $q_{1+J} \in Um(A_{1+J}),$ then there exists a $p \in Um(P)$ such that $p \equiv q$ modulo $A_+P$.
	\end{theorem}
	
	\begin{theorem}\label{mmt4}
		Let R be a normal ring of dimension $d,$ $M \in \mathfrak{M}_n$ a normal $\phi$-simplicial monoid of rank $r > 0$ and $A=R[M]$. Let $P$ be a projective $A$-module of rank $ dim(A) - n$ and $J=J(R,P)$ be the Quillen ideal of $P$. Assume
		\begin{enumerate}
			\item  $Um(P_f) \neq \emptyset$ for some $f \in R[M]$ monic in $t_1$;
			\item When $n>1,$ $M \in \mathfrak{M}_n$ is such that the automorphism $\widetilde{\eta}$ obtained has the form $\widetilde{\eta}(t_i) \in  t_i + M_1$ for $i > 1$.
		\end{enumerate}
		Then the map $Um(P) \rightarrow Um(P/A_+^1P)$ is surjective.
	\end{theorem}
	
	\begin{proof}
		We may assume $R$ to have a connected spectrum. Let $J=J(R,P) \subset R$. If $d=0,$ then $R$ is a field and by Theorem \ref{tp2}$,$ both $P$ and $P/A_+^1P$ are free. As $P$ is extended from $R[M_0^1]$ and $R[M] \rightarrow R[M_0^1]$ is a retraction$,$ we have the required surjection.  Let $d \geq 1$. Let ``$\sim$'' denote reduction modulo $JA$ and ``$-$'' denote reduction modulo $(J,A_+^1)$.
		
		Case 1: If $d=1,$ then $R$ is a regular ring. Let $\mathfrak{p} \in Spec(R)$. As $M$ is positive$,$ by (\cite{Swan-MR1144038}$,$ Theorem 1.2)$,$ $P_{\mathfrak{p}}$ is extended from $R_{\mathfrak{p}}$. By the graded version of Quillen's local-global principle  (\cite{Gubeladze1-MR2508056}$,$ Theorem 8.11)$,$ $P$ is extended from $A/A_+$ and hence from $A/A_+^1$. Since $R[M] \rightarrow R[M_0^1]$ is a retraction$,$ the map  $Um(P) \rightarrow Um(P/A_+^1P)$ is surjective.  
		
		Case 2: Let $d>1$. Next we want to prove $ht(J) \geq 2$. Let if possible $\mathfrak{p} \in Spec(R)$ be a height $1$ minimal prime of $J$. Then $R_{\mathfrak{p}}$ is a PID and by Theorem \ref{tp2}$,$ $P_{\mathfrak{p}}$ is free. This would imply the existence of $s \in  J \cap (R \smallsetminus \mathfrak{p}),$ a contradiction. Therefore $ht(J) \geq 2$. Let $p_2 \in Um(P/A_+^1P)$. As $M_0^1 \in \mathfrak{M}_{n-1},$ $dim(R/J) \leq d-2$ and $rank(\widetilde{P})=d+r-n > (d-2) + r - n,$ by  Theorem \ref{t6}$,$ the map $Um(\widetilde{P}) \rightarrow Um(\bar{P})$ is onto. Choose $p_1 \in Um(\widetilde{P})$ such that $\pi_1(p_1) = \pi_2(p_2)$. Consider the following patching diagram:
		
		\begin{figure}[H]
			\tikzset{column sep=small, row sep=small, ampersand replacement=\&}
			\centering
			\begin{floatrow}
				\ffigbox{\begin{tikzcd}
						A/JA_+^1A  \arrow[rr,"\theta_1"] \arrow[dd,"\theta_2"] \&\&
						A/JA  \simeq (R/JR)[M] \arrow[dd,"\pi_1"] \\
						\&  \\
						A/A_+^1 \simeq R[M_0^1] \arrow[rr, "\pi_2"] \&\& A/(J, A_+^1)A \simeq (R/JR)[M_0^1]. 
				\end{tikzcd}}{}
			\end{floatrow}
		\end{figure} 
		
		Then there exist $p' \in Um(P/JA_+^1P)$ such that $\theta_1(p') = p_1 \in Um(\widetilde{P})$ and $\theta_2(p') = p_2 \in Um(P/A_+^1P)$. Using $p_1 \in Um(\tilde{P})$ we may decompose $\widetilde{P} = \widetilde{A}p_1 \oplus Q$. Let $q \in P$ be such that $f \in O_P(q)$. Using the decomposition above write $\tilde{q} = (\tilde{a}p_1,q')$. As a consequence of Eisenbud-Evans in \cite{Plumstead-MR722004} there exists a transvection $\tilde{\tau} \in Aut(\tilde{P})$ such that $\tilde{\tau}(\tilde{q}) = (\tilde{a}p_1,q'')$ and $ht_{\tilde{A}_{\tilde{a}}}(O_Q(q{''})) \geq rank(P) - 1 = d - 1 + r - n$. From (\cite{Bhat-Roy-MR727374}$,$ Proposition 4.1) we may lift $\tilde{\tau}$ to $\tau \in Aut(P)$. On replacing $P$ by $\tau(P),$ we may assume $ht_{\tilde{A}_{\tilde{a}}}(O_Q(q')) \geq rank(P)-1 = d-1+r-n$. 
		
		\textit{Claim}: $O_Q(q')$ contains a monic in 
		$t_1$. 
		
		Let $\{\mathfrak{p}_1, \ldots, \mathfrak{p}_s\}$ be minimal primes of $O_Q(q')$ not containing $\tilde{a}$. Then $ht(\cap \mathfrak{p}_i) \geq ht_{\tilde{A}_{\tilde{a}}}(O_Q(q')) \geq rank(P)-1 = d-1 + r - n$. 	Let $M = \pz{}[W_1, \ldots, W_l] \subset \pz{}[t_1, \ldots, t_r]$. As $M$ is $\phi$-simplicial$,$ the first $r$ elements can be chosen to be $t_i^{s_i}$ for some $s_i \in \mathbb{Z}_{>0}$. Consider the composition of maps
		\begin{linenomath*}
			\[(R/JR)[X_{j_1}, \ldots, X_{j_n}] \xhookrightarrow{  \hspace{0.2cm} i_j  \hspace{0.2cm} } (R/JR)[X_1, \ldots, X_l] \xtwoheadrightarrow{\beta} (R/JR)[M] \simeq \tilde{A},\]
		\end{linenomath*}
		where $\beta(X_i)=W_i$ for all $i$. 
		If $n=1,$ then $\cap \mathfrak{p}_i$ contains monic in $t_1$ with coefficients in $R/J$. Therefore all minimal primes of $O_Q(q')$ and hence $O_Q(q')$ contains a monic in $t_1$ with coefficients in $R/J$.
		
		Let $n>1$. As $dim(R/JR) \leq d-2,$ by Lemma \ref{l1} there exists quasi-monic $g \in \cap \mathfrak{p}_i \cap (R/JR)[t_1^{p_1}, \ldots, t_n^{p_n}]$. Since $M \in \mathfrak{M}_n,$ there exists an $(R/JR)[t_1]$-automorphism $\tilde{\eta}$ such that $\tilde{\eta}(g) \in \eta(\cap  \mathfrak{p}_i)$ is monic in $t_1$ with coefficients in $(R/JR)[M_0^1]$. By (2)$,$ we may lift $\tilde{\eta}$ to $\eta \in Aut_{R[t_1]}(R[M])$ and replace $A$ by $\eta(A)$. If $\mathfrak{p}$ is a minimal prime ideal of $O_Q(q')$ containing $\tilde{a},$ then $O_{\tilde{P}}(\tilde{q}) \subset \mathfrak{p}$. As $\eta$ preserves monic in $t_1,$ $\eta(f)$ is again monic in $t_1$. Therefore all minimal primes of $O_Q(q')$ and hence $O_Q(q')$ contains a monic in $t_1$ with coefficients in $(R/J)[M_0^1]$ say $\tilde{g}$. Choose $g \in A$ to  be a monic lift of $\tilde{g}$.
		
		Let $p \in P$ be a lift of $p' \in P/JA_+^1P$. In the final step we shift $p$ to $p_0 = p + t_1^Ngq$ for a well chosen $N$ and show that $(p_0)_{1+A_+^1} \in Um(P_{1+A_+^1})$ and $(p_0)_{1+J(R[M_0^1],P)} \in Um(P_{1+J(R[M_0^1],P)})$. We arrive at our desired conclusion by invoking Theorem \ref{tp1}.
		
		We can choose $N$ large enough so that $O_p(p_0)$ contains a monic in $t_1,$ say $h$.  As $\widetilde{p'}= p_1$ we get $\tilde{q} = (\tilde{a}\tilde{p}, q')$. Thus $\widetilde{p_0} = ((1+t_1^N\tilde{g}\tilde{a})\widetilde{p'}, t_1^N\tilde{g}^Nq')$. Since $O_Q(q')$ contains $\tilde{g},$ we have $O_{\widetilde{P}}(p_1) \subset O_{\widetilde{P}}(\tilde{p_0})$. Which in turn implies $\widetilde{p_0} \in Um(\widetilde{P})$ and thus $(p_0)_{1+AJ} \in Um(P_{1+AJ})$. By (\cite{Gubeladze-umodrow1-MR1161570}$,$ Lemma 6.3) the extension $R[M_0^1] \rightarrow R[M_0^1]/(h)$ is integral and thus $(R/J)[M_0^1]/(O_P(p_0) \cap R[M_0^1]) \rightarrow (R/J)[M]/O_P(p_0)$ is integral. As $O_P(p_0)$ and $JA$ are comaximal in $A,$ we have $O_P(p_0) \cap R[M_0]$ and $JR[M_0^1]$ are comaximal in $R[M_0^1]$. Therefore $p_0 \in Um(P_{1+JR[M_0]}) \subset Um(P_{1+J(R[M_0^1],P)})$ (Note that $JR[M_0^1] \subset J(R[M_0],P)$). 
		
		As $(p_0)_{1+A_+^1} \in Um(P_{1+A_+^1}),$ by Theorem \ref{tp1}$,$ there exists $p_3 \in Um(P)$ such that $p_3 \equiv p_0$ modulo $A_+^1P$. As $p_0 \equiv p_2$ modulo $A_+^1P,$ the surjection follows. 
	\end{proof}
	
	Note that in all of the examples discussed before$,$ the automorphisms occurring out of $M$ being in $\mathfrak{M}_n$ satisfy the condition (2) above. The above theorem results in a slew of corollaries.
	
	\begin{corollary}
		Let R be a normal ring of dimension $d,$ $M \in C(\phi)$ a normal monoid of rank $r > 0$ and $A=R[M]$. Assume $P$ to be a projective $A$-module of rank $ d$. If $Um(P_f) \neq \emptyset$ for some $f \in R[M]$ monic in $t_1,$ then the map $Um(P) \rightarrow Um(P/A_+^1P)$ is surjective. 
	\end{corollary}
	
	\begin{proof}
		Since $M \in C(\phi) \subset \mathfrak{M}_r,$  by definition, for $2 \leq i \leq r,$ $\exists$ $c_i \in \mathbb{N}$ and $\widetilde{\eta} \in Aut_{R[t_1]}(R[t_1, \ldots, t_r])$ given by
		\begin{linenomath*}
			\[
			\widetilde{\eta}(t_i) \mapsto  t_i + t_1^{c_i}.
			\]
		\end{linenomath*}
		$M$ clearly satisfies the second condition of the hypothesis of Theorem \ref{mmt4} and thus the result follows.
	\end{proof}
	
	\begin{corollary}\label{c5}
		Let $R$ be a normal ring of dimension d$,$ $M \in \mathcal{C}(\phi)$ be a normal monoid of rank $r > 0$ and $A=R[M]$. Assume $P$ to be a projective $A$-module of rank d. If for each $i,$ there exists $f_i \in R[M]$ monic in $t_i$ such that $Um(P_{f_i}) \neq \emptyset,$
		then the map $Um(P) \rightarrow Um(P/A_+P)$ is surjective.
	\end{corollary}
	
	\begin{proof}
		This follows from the above corollary by inducting on the rank of the monoid. If $r=0,$ then $M=0$ and we are done. Let $r>0$ and $-$ denote reduction modulo $A_+^1$. By induction on $R[M_0^1],$ we get 
		\begin{linenomath*}
			\[Um(\bar{P}) \rightarrow Um(\bar{P}/(R[M_0^1])_+\bar{P}) \simeq Um(P/A_+P)\]
		\end{linenomath*}
		is surjective.  The previous corollary gives $Um(P) \rightarrow Um(\bar{P})$ is surjective. This proves the required.
	\end{proof}
	
	The following was proved in (\cite{BLR-MR796196}$,$ Theorem 5.1) when $M$ is free:
	
	\begin{corollary}\label{e1}
		Let $R$ be a normal ring of dimension $d$. Let $ B$ be a birational overring of $R[X]$ and $M$ be a normal monoid of rank $2$. Then $S$-$dim(B[M]) \leq d$.
	\end{corollary}
	\begin{proof}
		Let $P$ be a projective $B[M]$-module of rank $ d + 1$. As $M$ is $\phi$-simplicial$,$ there exists $p_i \in \mathbb{Z}_{>0},$ such that $t_i^{p_i} \in M,$ for $ i =1, 2$. Choose $f_i = t_i^{p_i}$ for all $i$. Since $\mathbb{C}[f_i^{-1}M] \simeq \mathbb{C}[M]_{f_i}$ is normal$,$ from (\cite{Gubeladze1-MR2508056}$,$ Theorem 4.40) we may infer $f_i^{-1}M$ is a normal monoid for all $i$. From (\cite{Gubeladze1-MR2508056}$,$ Proposition 2.26) we can further deduce that $f_i^{-1}M \simeq \nz{} \oplus \pz{}$. By (\cite{BLR-MR796196}$,$ Theorem 5.1)$,$ $Um(P_{f_i}) \neq \emptyset,$ for $i=1, 2$. Thus using Corollary \ref{c5} we can conclude $Um(P) \neq \emptyset$. 
	\end{proof}
	
	\begin{corollary}\label{c1}
		Let R be a ring of dimension $d>1,$ $M \in \mathfrak{M}_n$ a normal $\phi$-simplicial monoid of rank $r > 0$ and $A=R[M]$. Let $P$ be a projective $A$-module of rank $ dim(A) - n,$ and  $J=J(R,P)$ be the Quillen ideal of $P$ of $height(J) > 1$. Assume
		\begin{enumerate}
			\item  $Um(P_f) \neq \emptyset$ for some $f \in R[M]$ monic in $t_1$;
			\item When $n>1,$ $M \in \mathfrak{M}_n$ is such that the automorphism $\widetilde{\eta}$ obtained has the form $\widetilde{\eta}(t_i) \in  t_i + M_1$ for $i > 1$.
		\end{enumerate}
		Then the map $Um(P) \rightarrow Um(P/A_+^1P)$ is surjective.
	\end{corollary}
	
	\begin{proof}
		This is a restatement of Theorem \ref{mmt4}, except here the restriction of $R$ normal is removed and additional conditions of $ht(J) > 1$ and $d>1$ is added. In proof of Theorem \ref{mmt4}, the normality of $R$ is used in two places, first in the case $d=1,$ and then in case $d \geq 2$ to show $ht(J) \geq 2$. The rest is identical to the proof of Theorem \ref{mmt4}.
	\end{proof}
	
	The following (also proved in (\cite{MKK-ZinnaMR3874658}$,$ Corollary 3.1)) is a consequence of the above corollary:
	
	\begin{corollary}
		Let $R$ be a ring of dimension $d$ and $A=R[T_1, \ldots, T_n]$. If $P$ is a projective $A$-module of rank $d$ such that $Um(P_f) \neq \emptyset$ for some $f \in A$ monic in $T_1$. If $Um(P/A_+^1P) \neq \emptyset,$ then $Um(P)$.
	\end{corollary}
	
	\begin{proof}
		We may assume $R$ is a reduced ring. If $d = 1,$ then $P_f$ is free. Applying (\cite{Quillen-MR427303}, Theorem 3) and (\cite{Suslin-MR0469905}, Theorem 1) to the ring $(R[T_2, \ldots, T_n])[T_1],$ we get $P$ is free. Let $d=2$. As $rank(P)=2,$ by (\cite{Bhatwadekar-birational-MR978296}, Proposition 3.3), $Um(P) \neq \emptyset$. 
		
		Let $d \geq 3$ and $A'=sn(R)[T_1, \ldots, T_n]$. Denote by $D$ the determinant of $P,$ $P'=P \otimes A'$ and by $J'=J(sn(R), P')$ the Quillen ideal of $P'$. Choose $\mathfrak{p} \in Spec(sn(R))$ of height $1$ and the multiplicative subset $S=sn(R) \smallsetminus \mathfrak{p}$. Further, $dim({sn(R)}_{\mathfrak{p}}) = 1$ and $S^{-1}A' = {sn(R)}_{\mathfrak{p}}[T_1, \ldots, T_n]$.  Since $rank(P) \geq 3,$
		\begin{linenomath*}
			\[
			S^{-1}(P') \simeq S^{-1}(D \otimes A') \oplus (S^{-1}A')^{d-1},
			\]
		\end{linenomath*}
		by \cite{Bhat-Roy-MR727374}. By \cite{Swan-MR595029}, $D \otimes A'$ is extended from $sn(R)$.  This in turn implies  $S^{-1}(D \otimes A')$ is extended from the local ring $sn(R)_{\mathfrak{p}}$ and therefore $ S^{-1}(D \otimes A') = S^{-1}A'.$ Therefore $ht(J') \geq 2$. Assume $ Um(P/A_+^1P) \neq \emptyset$. As
		\begin{linenomath*}
			\[
			P/A_+^1P \simeq P \otimes R[T_2, \ldots, T_n] 
			\]
		\end{linenomath*}
		\begin{linenomath*}
			\[
			P'/{A'}_+^1P' \simeq P \otimes sn(R)[T_2, \ldots T_n],
			\]
		\end{linenomath*}
		we have $Um(P'/ {A'}_+^1P') \neq \emptyset$.  Apply hypothesis of Corollary \ref{c1} to the ring $sn(R)$ and $M=\pz{}[T_1, \ldots, T_n]$. Since conditions (1) and (2) hold, we get $Um(P') \neq \emptyset$. By Lemma \ref{pl1} and Theorem \ref{pt1}, $Um(P) \neq \emptyset$.
	\end{proof}

	\begin{nolinenumbers}	
		\bibliographystyle{abbrv}
		\bibliography{paper2.bib}
	\end{nolinenumbers}
	
\end{document}